\documentclass[11pt]{article}
\usepackage{amsfonts,latexsym,rawfonts,amsmath,amssymb,amsthm,graphicx}
\numberwithin{equation}{section}
\newtheorem{theorem}{Theorem}[section]

\newtheorem{thm}[theorem]{Theorem}
\newtheorem{pro}[theorem]{Proposition}
\newtheorem{cor}[theorem]{Corollary}

\newtheorem{rem}[theorem]{Remark}

\title{Universal inequalities for eigenvalues of the Dirichlet Laplacian on conformally flat Riemannian manifolds}
\author{Yong Luo, Xianjing Zheng}
\date{}
\begin{document}
	\maketitle
	\begin{abstract} In this paper we study eigenvalues of the Dirichlet Laplacian on conformally flat Riemannian manifolds. In particular we establish some universal inequality for eigenvalues of the Dirichlet Laplacian on the hyperbolic space $\mathbb{H}^n(-1)$.
	\end{abstract}
	\section{Introduction}
Let $(M^n, g)$  be an $n$-dimensional complete Riemannian manifold and $\Omega$  a bounded connected domain in $M$. Let  $\Delta$ denote the Laplacian  acting on functions on $ M $. Eigenvalues of the Dirichlet Laplacian is stated as:
     	\begin{eqnarray}\label{Dirichlet problem}
     	\left\{\begin{array}{c}
     		\Delta u=-\lambda u \text { in } \Omega, \\
     		u|_{\partial \Omega}=0 .
     	\end{array}\right.
     \end{eqnarray}
     \indent
     Let
     \begin{eqnarray*}
     	0<\lambda_{1}<\lambda_{2} \leq \lambda_{3} \leq \cdots
     \end{eqnarray*}
     denote the successive eigenvalues of (\ref{Dirichlet problem}). Here each eigenvalue is repeated according to its multiplicity. It is well-known that the following Weyl's asymptotic formula holds (cf. \cite{Cha}):
\begin{eqnarray}\label{Weyl asy.}
\lambda_k\sim\frac{4\pi^2}{(\omega_n Vol\Omega)^\frac{2}{n}}k^\frac{2}{n}, \ k\to \infty,
\end{eqnarray}
where $\omega_n$ is the volume of the unit ball in $\mathbb{R}^n$.\\
     \indent
Generally, if an equality of eigenvalues requires no hypotheses on the geometric quantities of the domain $\Omega$ (other than its dimension), it is called an universal inequality. Universal inequality is an important research area in the study of the spectrum of $\Delta$ and many works have been done during the past decades. Now we give a brief introduction of the main results.\\
     \indent
     The study of universal inequalities for (\ref{Dirichlet problem}) was initiated by Payne, P\"olya, and Weinberger \cite{PPW} in 1956, they proved
     \begin{eqnarray*}
     	\lambda_{k+1}-\lambda_{k} \leq \frac{2}{k} \sum_{i=1}^{k} \lambda_{i}
     \end{eqnarray*}
     for $\Omega \subset \mathbb{R}^{2} $. Exactly by the same proof, their result can promoted to dimension $n$ as
     \begin{eqnarray*}
     	\lambda_{k+1}-\lambda_{k} \leq \frac{4}{k n} \sum_{i=1}^{k} \lambda_{i}
     \end{eqnarray*}
     for $ \Omega \subset \mathbb{R}^{n}$.\\
     \indent
     In 1980, Hile and Protter \cite{HP} proved
     \begin{eqnarray*}
      \sum_{i=1}^{k} \frac{\lambda_{i}}{\lambda_{k+1}-\lambda_{i}} \geq \frac{k n}{4}.
     \end{eqnarray*}\\
     \indent
      In 1991, Yang \cite{Yang} (cf. \cite{CY1'}) proved a sharp universal inequality:
     \begin{eqnarray}\label{Yang's ineq.}
     	\sum_{i=1}^{k}\left(\lambda_{k+1}-\lambda_{i}\right)^2\leq\frac{4}{n}\sum_{i=1}^{k}\left(\lambda_{k+1}-\lambda_i\right) \lambda_{i},
     \end{eqnarray}
     which has been  called Yang's first inequality by Ashbaugh (cf. \cite{Ash1} and \cite{Ash2} and so on).

\indent
When $\Omega$ is a bounded domain in $\mathbb{S}^{n}(1)$, Cheng and Yang \cite{CY1} proved in 2005 that eigenvalues of the eigenvalue problem (\ref{Dirichlet problem}) satisfy
\begin{eqnarray}\label{Cheng's 2}
		\sum_{i=1}^{k}\left(\lambda_{k+1}-\lambda_{i}\right)^{2} \leq\frac{4}{n} \sum_{i=1}^{k}\left(\lambda_{k+1}-\lambda_{i}\right)\left(\lambda_i+\frac{n^2}{4}\right),
\end{eqnarray}
by canonically embedded $\mathbb{S}^{n}(1)$  into $\mathbb{R}^{n+1}$ (see also \cite{Chen} and \cite{Sou}). It is optimal since the above inequality becomes an equality for any $k$ when $\Omega=\mathbb{S}^n(1)$.

It is very natural to consider universal inequalities for eigenvalues of the eigenvalue problem (\ref{Dirichlet problem}) when $M^n$ is the hyperbolic space $\mathbb{H}^{n}(-1)$. If $n=2$, by making use of estimates for eigenvalues of the eigenvalue
problem of the Schr\"odinger like operator with a weight, Harrell and Michel \cite{HM} and Ashbaugh \cite{Ash2}
have obtained several results. For general $n$, Cheng and Yang \cite{CY2} proved in 2009 that eigenvalues of the eigenvalue problem (\ref{Dirichlet problem}) satisfy
\begin{eqnarray}\label{Cheng's 1}
	\sum_{i=1}^{k}\left(\lambda_{k+1}-\lambda_{i}\right)^{2} \leqslant 4 \sum_{i=1}^{k}\left(\lambda_{k+1}-\lambda_{i}\right)\left(\lambda_{i}-\frac{(n-1)^{2}}{4}\right).
\end{eqnarray}

It remains a challenge problem if one can improve the coefficient on the right hand side of $(\ref{Cheng's 1})$ from $4$ to $\frac{4}{n}$ (which is sharp by making use of a recursion formula of Cheng and Yang \cite{CY1'} and Weyl's asymptotic formula (\ref{Weyl asy.})). Actually we have a conjecture raised by Cheng \cite{Che}.
\\
\\\textbf{Conjecture:} Assume that $\lambda_i$ are eigenvalues of the eigenvalue problem (\ref{Dirichlet problem}) with $(M^n, g)$ being the hyperbolic space $\mathbb{H}^{n}(-1)$. Then have we
\begin{eqnarray*}
	\sum_{i=1}^{k}\left(\lambda_{k+1}-\lambda_{i}\right)^{2} \leqslant \frac{4}{n} \sum_{i=1}^{k}\left(\lambda_{k+1}-\lambda_{i}\right)\left(\lambda_{i}-\frac{(n-1)^{2}}{4}\right).
\end{eqnarray*}	

In this paper, we consider universal inequalities for the eigenvalues of the Dirichlet Laplacian on the hyperbolic space $\mathbb{H}^{n}(-1)$.  More generally, we consider eigenvalues of the eigenvalue problem (\ref{Dirichlet problem}) on conformally flat manifolds. We have
	\begin{pro}\label{theorem 1}
		Let $(\mathbb{R}^{n},\widetilde{g})$ be an $n$-dimensional conformally flat Riemannian manifold with the conformal metric $\widetilde{g}=e^{2f}g$, where $f$ is a smooth function on $\mathbb{R}^{n}$, and let $\Omega$ be a bounded smooth domain in $(\mathbb{R}^{n},\widetilde{g})$. Then eigenvalues $\lambda_{i}$'s of the eigenvalue problem (\ref{Dirichlet problem}) satisfy
	\begin{eqnarray}\label{the first result}
\sum_{i=1}^{k}\left(\lambda_{k+1}-\lambda_{i}\right)^{2} \int_{\Omega} x_{p} u_{i}\left(-2\widetilde{\nabla}x_{p} \widetilde{\nabla}u_{i}-(n-2) e^{-2 f} \nabla f \nabla x_{p} u_{i}\right) \nonumber\\
\leqslant \sum_{i=1}^{k}\left(\lambda_{k+1}-\lambda_{i}\right) \int_{\Omega}\left|2\widetilde{\nabla}x_{p} \widetilde{\nabla}u_{i}+(n-2) e^{-2 f} \nabla f \nabla x_{p} u_{i}\right|^{2}
	\end{eqnarray}
	where $x_{p}$ $ (p=1,2,\cdots,n)$ denotes the standard coordinates of  $\mathbb{R}^{n}$ , while $\widetilde{\nabla}$ and $\nabla$ denotes the gradient of $(\mathbb{R}^{n},\widetilde{g})$ and $(\mathbb{R}^{n},g)$, respectively.
	\end{pro}
	From Proposition \ref{theorem 1} we obtain
	\begin{pro}\label{Yang inequality in hyperbolic space}
     For a conformally flat manifold $(\mathbb{R}_{+}^{n},\widetilde{g})$ with conformal metric $$\widetilde{g}=\frac{d x_{1}^{2}+d x_{2}^{2}+\cdots+d x_{n}^{2}}{x_{n}^{t}},$$ where $t$ is a real constant, eigenvalues $\lambda_{i}$'s of the eigenvalue problem (\ref{Dirichlet problem}) satisfy
     \begin{align}\label{the second result}
     	& \sum_{i=1}^{k}\left(\lambda_{k+1}-\lambda_{i}\right)^{2}\left(n \int_{\Omega} x_{n}^{t} u_{i}^{2}\right) \nonumber\\
     \leqslant & \sum_{i=1}^{k}\left(\lambda_{k+1}-\lambda_{i}\right)\left(4 \lambda_{i} \int_{\Omega} x_{n}^{t} u_{i}^{2}+\left(\left(-\frac{n^2}{4}+n+1\right)t^{2}-nt\right)\int_{\Omega} x_{n}^{2 t-2} u_{i}^{2}\right).
     \end{align}
     Here $\mathbb{R}_{+}^n:=\{(x_1,\cdots,x_n)\in \mathbb{R}^n | x_n>0\}$.
	\end{pro}
	Recall that $(\mathbb{R}_{+}^{n},\widetilde{g})$ with the conformal metric $$\widetilde{g}=\frac{d x_{1}^{2}+d x_{2}^{2}+\cdots+d x_{n}^{2}}{x_{n}^{2}}$$ is the upper half space model of the hyperbolic space $\mathbb{H}^{n}(-1)$. Therefore let $t=2$ and $k=1$ in Proposition \ref{Yang inequality in hyperbolic space} we obtain the following result.
	\begin{thm}\label{case t=2}
		Assume that $(M^n,g )$ is the hyperbolic space $\mathbb{H}^{n}(-1)=(\mathbb{R}_{+}^{n}, \frac{d x_{1}^{2}+d x_{2}^{2}+\cdots+d x_{n}^{2}}{x_{n}^{2}})$, then eigenvalues $\lambda_{1}, \lambda_2$ of the eigenvalue problem (\ref{Dirichlet problem}) satisfy
		\begin{eqnarray}\label{the third result}
			\lambda_2-\lambda_1\leqslant \frac{4}{n}\left(\lambda_1-\frac{n^{2}-2n-4}{4}\right).
		\end{eqnarray}
	\end{thm}
The hyperbolic space $\mathbb{H}^n(-1)$  can also be seem as a conformally flat Riemannian manifold with a radial symmetric conformal factor (the Poincar\'e unit disk model). From this aspect, we can give an alternative proof of Theorem \ref{case t=2}. First we prove
	{\begin{pro}\label{Yang type inequality for radial conformal factor}
 Assume that $(M^n, \tilde{g})$ is an $n$-dimensional conformally flat manifold $M$ with metric $\widetilde{g}=e^{2f}g$, where $g$ is the standard metric in $\mathbb{R}^n$, and $f$ is a radial symmetric function. Then eigenvalues $\lambda_{i}$'s of the eigenvalue problem (\ref{Dirichlet problem}) satisfy
		 \begin{align}\label{the radial conformal factor case}
		 	& \sum_{i=1}^{k}\left(\lambda_{k+1}-\lambda_{i}\right)^{2}\left(n \int_{\Omega} e^{-2 f} u_{i}^{2}\right) \nonumber\\
		 \leqslant & \sum_{i=1}^{k}\left(\lambda_{k+1}-\lambda_{i}\right)\left(4 \lambda_{i} \int_{\Omega} e^{-2f} u_{i}^{2}-\left(n^{2}-4 n-4\right) \int_{\Omega} e^{-4 f}\left|f^{\prime}(r)\right|^{2} u_{i}^{2}\right. \nonumber\\
		 	& \left.-2 n \int_{\Omega} e^{-4 f}\left(f^{\prime\prime}(r)+\frac{f^{\prime}(r)(n-1)}{r}\right) u_{i}^{2}\right),
		 \end{align}
		where $r:=\left|x\right|$.
	\end{pro}
	Recall that the Poincar\'e unit disk model of the hyperbolic space $$\mathbb{H}^n(-1)=(\mathbb{B}^n, \frac{4}{(1-|x|^2)^2}g),$$
 where $\mathbb{B}^n$ is the unit disk in $\mathbb{R}^n$ and $g=dx^2_1+\cdots+dx_n^2$ is the canonical metric on $\mathbb{B}^n$. As an application of (\ref{the radial conformal factor case}), we have
	\begin{cor}\label{the second proof}
	 Assume that $(M^n, g)$ is the hyperbolic space $\mathbb{H}^{n}(-1)=(\mathbb{B}^n, \frac{4}{(1-|x|^2)^2}g)$, then eigenvalues $\lambda_{i}$'s of the eigenvalue problem (\ref{Dirichlet problem}) satisfy
		 \begin{align}\label{radial case}
		 	& \sum_{i=1}^{k}\left(\lambda_{k+1}-\lambda_{i}\right)^{2}\left(n \int_{\Omega}\left(1-|x|^{2}\right)^{2} u_{i}^{2}\right) \nonumber\\
		 	\leqslant & \sum_{i=1}^{k}\left(\lambda_{k+1}-\lambda_{i}\right)\big[4\lambda_{i} \int_{\Omega}\left(1-|x|^{2}\right)^{2}u_{i}^{2}-{n}^{2}\int_{\Omega}\left(1-|x|^{2}\right)^{2}u_{i}^{2} \nonumber\\
&+\left(2n+4\right) \int_{\Omega}\left(1-|x|^{2}\right)^{2}|x|^{2} u_{i}^{2}\big].
		 \end{align}
	\end{cor}
Note that since $|x|<1$, we see that $$\int_{\Omega}\left(1-|x|^{2}\right)^{2}|x|^{2} u_{i}^{2}\leq\int_{\Omega}\left(1-|x|^{2}\right)^{2} u_{i}^{2},$$ and thus we can deduce (\ref{the third result})  from (\ref{radial case}) by letting $k=1$.
\vspace{0.2cm}

Let $\rho_{max}:=\max_{\Omega}\frac{1}{x_n^2}$ and $\rho_{min}:=\min_\Omega\frac{1}{x_n^2}$. For general $k$, it easy to see that we can deduce the following result from Proposition \ref{Yang inequality in hyperbolic space}.
\begin{thm}
		Assume that $(M^n,g )$ is the hyperbolic space $\mathbb{H}^{n}(-1)=(\mathbb{R}_{+}^{n}, \frac{d x_{1}^{2}+d x_{2}^{2}+\cdots+d x_{n}^{2}}{x_{n}^{2}})$, then eigenvalues $\lambda_{i}$'s of the eigenvalue problem (\ref{Dirichlet problem}) satisfy
		\begin{eqnarray}
			\sum_{i=1}^k(\lambda_{k+1}-\lambda_i)^2\leqslant \frac{\rho_{max}}{\rho_{min}}\frac{4}{n}\sum_{i=1}^k\left(\lambda_{k+1}-\lambda_i\right)\left(\lambda_i-\frac{n^{2}-2n-4}{4}\right).
		\end{eqnarray}
	\end{thm}
\begin{rem}
Note that Ashbaugh \cite{Ash2} proved similar result for eigenvalues of the eigenvalue
problem of the Schr\"odinger like operator with a weight, and as a consequence, he proved that eigenvalues of the eigenvalue problem (\ref{Dirichlet problem}) satisfy \begin{eqnarray*}
			\sum_{i=1}^k(\lambda_{k+1}-\lambda_i)^2\leqslant \frac{\rho_{max}}{\rho_{min}}2\sum_{i=1}^k\left(\lambda_{k+1}-\lambda_i\right)\lambda_i,
		\end{eqnarray*}
when $\Omega$ is a bounded domain in $\mathbb{H}^{2}(-1)$. Our result generalize his result to general dimension.
\end{rem}
Assume that $\Omega\subset\{0<a\leq x_n^2\leq b\}$, then we have
\begin{cor}
  Assume that $(M^n,g )$ is the hyperbolic space $\mathbb{H}^{n}(-1)=(\mathbb{R}_{+}^{n}, \frac{d x_{1}^{2}+d x_{2}^{2}+\cdots+d x_{n}^{2}}{x_{n}^{2}})$, then eigenvalues $\lambda_{i}$'s of the eigenvalue problem (\ref{Dirichlet problem}) satisfy
		\begin{eqnarray}
			\sum_{i=1}^k(\lambda_{k+1}-\lambda_i)^2\leqslant \frac{b}{a}\frac{4}{n}\sum_{i=1}^k\left(\lambda_{k+1}-\lambda_i\right)\left(\lambda_i-\frac{n^{2}-2n-4}{4}\right).
		\end{eqnarray}
\end{cor}
\textbf{Organization.} In the preliminaries we collect some useful formulas which will be used many times in this paper. Proposition \ref{theorem 1} will be proved in section 3, Proposition \ref{Yang inequality in hyperbolic space} will be proved in section 4, while Proposition \ref{Yang type inequality for radial conformal factor} and Corollary \ref{the second proof} will be proved in section 5. In the last section we discuss more consequences of Propositions \ref{theorem 1} and \ref{Yang type inequality for radial conformal factor}.

	\section{Preliminaries}

In this paper, we denote $$\nabla f \nabla h=\langle \nabla f, \nabla h\rangle_{g}, \ \widetilde{\nabla} f \widetilde{\nabla} h=\langle \widetilde{\nabla}f, \widetilde{\nabla}h\rangle_{\widetilde{g}}$$ for any functions $f,h$  $:$ $\Omega$ $\rightarrow$ $\mathbf{R}$. Then we have
\begin{align*}
\widetilde{\nabla} f \widetilde{\nabla} h
&=\langle\widetilde{\nabla} f, \widetilde{\nabla} h\rangle_{\widetilde{g}}
	\\&=e^{2 f}\left\langle e^{-2 f} \nabla f, e^{-2 f} \nabla h\right\rangle_{g}
	\\&=e^{-2 f}\langle\nabla f, \nabla h\rangle_{g}
	\\&=e^{-2 f} \nabla f \nabla h.
\end{align*}

	We now review some formulas of the Laplacian under conformal change of Riemannian metrics. Let $(M,g)$ be an $n$-dimensional Riemannian manifold. Given a function $h$$:$ $M$$ \rightarrow$ $\mathbb{R}$, we consider  the conformal manifold $(M,\widetilde{g})$ with the conformal metric $\widetilde{g}=e^{2f}g$. Let $\widetilde{\Delta}$ and $\Delta$ denote the Laplacian on  $(M,\widetilde{g})$ and $(M,g)$ respectively and $\widetilde{\nabla}$,  $\nabla$ denote the Riemannian connection of $(M,\widetilde{g})$ and $(M,g)$ respectively. There are known formula of the Laplacian under conformal transformation, which states that for a smooth function $F$ $:$ $M$ $\rightarrow$ $\mathbb{R}$,
	\begin{eqnarray}
		\quad \widetilde{\Delta} F=e^{-2 f}\left(\Delta F+(n-2) \nabla f  \nabla F\right).
	\end{eqnarray}
	It is obvious that for any smooth function $G$ $:$ $M$ $\rightarrow$ $\mathbb{R}$ we have,
	\begin{eqnarray}\label{multiply function}
		\widetilde{\Delta}(FG)=\widetilde{\Delta}FG+\widetilde{\Delta}GF+2\widetilde{\nabla }F\widetilde{\nabla}G\nonumber
=\widetilde{\Delta}FG+\widetilde{\Delta}GF+2e^{-2f}\nabla F \nabla G.
	\end{eqnarray}

\section{Proof of Proposition \ref{theorem 1}}
\begin{proof}
	Let $\lambda_{i}$, i=1, \ldots , k+1, be the $ith$ eigenvalue of problem (\ref{Dirichlet problem}) and  $u_{i}$  the orthonormal eigenfunction corresponding to  $\lambda_{i}$ , that is,
	\begin{eqnarray*}
		\widetilde{\Delta} u_{i}=-\lambda_{i} u_{i} \quad \text { in } \Omega,\left.\quad u_{i}\right|_{\partial \Omega}=0, \quad \int_{M} u_{i} u_{j}=\delta_{i j}, \quad \forall i, j=1,2, \ldots,k+1.
	\end{eqnarray*}
  Let $x_{1}, x_{2}, \ldots, x_{n}$ denote the standard coordinates in $\mathbb{R}^{n}$, for $i=1, \ldots, k$. Consider the trail functions  $\phi_{i} : \Omega \rightarrow\mathbb{R}$  given by
  \begin{eqnarray*}
  	\phi_{i}=x_{p} u_{i}-\sum_{j=1}^{k} a_{i j} u_{j},
  \end{eqnarray*}
  where
  \begin{eqnarray*}
  	a_{i j}=\int_{\Omega} x_{p} u_{i} u_{j} .
  \end{eqnarray*}
  Since  $\left.\phi_{i}\right|_{\partial \Omega}=0 $ and
  \begin{eqnarray*}
  	\int_{\Omega} u_{j} \phi_{i}=0, \forall i, j=1, \ldots, k,
  \end{eqnarray*}
using (\ref{multiply function}) we obtain
  \begin{align*}
  	\widetilde{\Delta} \phi_{i} & =\widetilde{\Delta}\left(x_{p} u_{i}-\sum_{j=1}^{k} a_{i j} u_{j}\right) \\
  	& =\widetilde{\Delta} x_{p}  u_{i}+\widetilde{\Delta} u_{i}  x_{p}+2 e^{-2 f} \nabla x_{p}  \nabla u_{i}+\sum_{j=1}^{k} a_{ij} \lambda_{j} u_{j}.
  \end{align*}
  Hence we have
  \begin{eqnarray}\label{fenzi}
  	\int_{\Omega} \phi_{i}\left(-\widetilde{\Delta}\phi_{i}\right)
  	&=&\int_{\Omega} \phi_{i}\left(-\widetilde{\Delta} x_{p} u_{i}-\widetilde{\Delta} u_{i}  x_{p}-2 e^{-2 f} \nabla x_{p}  \nabla u_{i}-\sum_{j=1}^{k} a_{ij} \lambda_{j} u_{j}\right)\nonumber \\
  	&=&\int_{\Omega} \phi_{i}\left(\lambda_{i} u_{i} x_{p}-\lambda_{i}\sum_{j=1}^{k} a_{ij}  u_{j}-\widetilde{\Delta} x_{p} u_{i}-2 e^{-2 f} \nabla x_{p}  \nabla u_{i}\right)\nonumber \\
  	&=&\lambda_{i}\left\|\phi_{i}\right\|^{2}-\int_{\Omega} \phi_{i}\left(\widetilde{\Delta} x_{p}  u_{i}+2 e^{-2 f} \nabla x_{p}  \nabla u_{i}\right), \nonumber\\
  \end{eqnarray}
  where  $\left\|\phi_{i}\right\|^{2}=\int_{\Omega} \phi_{i}^{2}$. It follows from the Rayleigh-Ritz inequality that
  \begin{align}\label{Rayleigh Rize inequality}
  	&\left(\lambda_{k+1}-\lambda_{i}\right)|| \phi_{i} \|^{2} \nonumber\\
  	&\quad\leq \int_{\Omega} -\left(\widetilde{\Delta} x_{p} u_{i}+2 e^{-2 f} \nabla x_{p}  \nabla u_{i}\right)\left(x_{p} u_{i}-\sum_{j=1}^{k} a_{i j} u_{j}\right) \nonumber\\
  	&\quad=\int_{\Omega} x_{p} u_{i} q_{i}+\sum_{j=1}^{k} a_{i j} \int_{\Omega} \widetilde{\Delta} x_{p} u_{i} u_{j}+2 \sum_{j=1}^{k} a_{ij} \int_{\Omega} e^{-2 f} \nabla x_{p} \nabla u_{i} u_{j} \nonumber\\
  	&\quad=\int_{\Omega} x_{p} u_{i}  q_{i}+(n-2) \sum_{j=1}^{k} a_{i j} \int_{\Omega} e^{-2 f} \nabla f \nabla x_{p} u_{i} u_{j}+2 \sum_{j=1}^{k} a_{i j} \int_{\Omega} e^{-2 f} \nabla x_{p} \nabla u_{i} u_{j}, \nonumber \\
  \end{align}
  where $q_{i}=-\left(\widetilde{\Delta} x_{p}  u_{i}+2 e^{-2 f} \nabla x_{p}  \nabla u_{i}\right)$. Now let
  \begin{align*}
  	s_{i j}=\int_{\Omega} e^{-2 f} \nabla f \nabla x_{p} u_{i} u_{j},\\
  	b_{i j}=\int_{\Omega} e^{-2 f} \nabla x_{p} \nabla u_{i} u_{j}.\\
  \end{align*}
  Note that
  \begin{align*}
  	&\lambda_{j} a_{i j} =  \int_{\Omega} x_{p} u_{i}\left(-\widetilde{\Delta} u_{j}\right)\nonumber\\
  	&\qquad=  -\int_{\Omega} u_{j} \widetilde{\Delta}\left(x_{p} u_{i}\right) \nonumber\\
  	&\qquad = -\int_{\Omega} u_{j}\left(\widetilde{\Delta} x_{p} u_{i}+x_{p} \widetilde{\Delta} u_{i}+2 e^{-2 f} \nabla x_{p} \nabla u_{i}\right) \nonumber\\
  	&\qquad= -(n-2) \int_{\Omega} e^{-2 f} \nabla f \nabla x_{p} u_{i} u_{j}+\lambda_{i} a_{i j}-2 b_{i j} \nonumber\\
  	&\qquad =  -(n-2) s_{i j}+\lambda_{i} a_{i j}-2 b_{i j}, \nonumber\\
  	\end{align*}
  thus we have $$ b_{i j}=\frac{\left(\lambda_{i}-\lambda_{j}\right)}{2} a_{i j}-\frac{(n-2)}{2} s_{i j},$$ and
  	\begin{align}\label{Rayleigh Rize inequality2}
  	\left(\lambda_{k+1}-\lambda_{i}\right)\left\|\phi_{i}\right\|^{2} &\leqslant \int_{\Omega} x_{p} u_{i} q_{i}+(n-2) \sum_{j=1}^{k} a_{i j} s_{i j}+2 \sum_{j=1}^{k}\left(\frac{\left(\lambda_{i}-\lambda_{j}\right)}{2} a_{i j}-\frac{(n-2)}{2} s_{i j}\right) a_{i j} \nonumber\\
  	 & = \int_{\Omega} x_{p} u_{i} q_{i}+\sum_{j=1}^{k}\left(\lambda_{i}-\lambda_{j}\right) a_{i j}^{2}.
  	\end{align}
Set $r_{i j}=\int_{\Omega} u_{j}\left(2 \widetilde{\nabla} x_{p} \widetilde{\nabla} u_{i}+\widetilde{\Delta} x_{p} u_{i}\right),$ then $r_{i j}=-r_{j i}$ since
  \begin{align*}
  r_{i j}&=2\int_\Omega e^{-2 f} \nabla x_{p} \nabla u_{i} u_{j}+(n-2)\int_{\Omega} e^{-2 f} \nabla f \nabla x_{p} u_{i} u_{j}
  \\&=2b_{i j}+(n-2)s_{i j}
  \\&=(\lambda_i-\lambda_j)a_{i j}.
  \end{align*}
  Furthermore,
  \begin{align}\label{key function}
   &\int_{\Omega}(-2) \phi_{i} \left(2  \widetilde{\nabla} x_{p} \widetilde{\nabla} u_{i}+\widetilde{\Delta} x_{p} u_{i}\right) \nonumber\\
    &= -2 \int_{\Omega} x_{p} u_{i}\left(2 \widetilde{\nabla} x_{p} \widetilde{\nabla} u_{i}+\widetilde{\Delta} x_{p} u_{i}\right)+2 \sum_{j=1}^{k} a_{i j} r_{i j} \nonumber\\
  	&=w_{i}+2 \sum_{j=1}^{k} a_{i j} r_{i j},
  \end{align}
  where $w_{i}=-2 \int_{\Omega} x_{p} u_{i}\left(2  \widetilde{\nabla} x_{p} \widetilde{\nabla} u_{i}+\widetilde{\Delta} x_{p} u_{i}\right).$

  Multiplying (\ref{key function}) by $(\lambda_{k+1}-\lambda_{i})^2$ and using the Schwarz inequality and (\ref{Rayleigh Rize inequality2}), we get
  \begin{align}
  	&\quad\left(\lambda_{k+1}-\lambda_{i}\right)^{2}\left(w_{i}+2 \sum_{j=1}^{k} a_{i j} r_{i j}\right)\nonumber\\
  	&=\left(\lambda_{k+1}-\lambda_{i}\right)^{2}\int_{\Omega}(-2) \phi_{i} \left(2 \widetilde{\nabla} x_{p} \widetilde{\nabla} u_{i}+\widetilde{\Delta} x_{p} u_{i}\right)\nonumber\\
  	&\leqslant\left(\lambda_{k+1}-\lambda_{i}\right)^{3}\int_{\Omega}\left|\phi_{i}\right|^{2}+\left(\lambda_{k+1}-\lambda_{i}\right) \int_{\Omega}\left|2  \widetilde{\nabla} x_{p} \widetilde{\nabla} u_{i}+\widetilde{\Delta} x_{p} u_{i}-\sum_{j=1}^{k} r_{i j} u_{j}\right|^{2} \nonumber\\
  	&\leqslant\left(\lambda_{k+1}-\lambda_{i}\right)^{2} \left(\int_{\Omega} x_{p} u_{i} q_{i}+\sum_{j=1}^{k}\left(\lambda_{i}-\lambda_{j}\right) a_{i j}^{2}\right)\nonumber\\
  	&\quad+\left(\lambda_{k+1}-\lambda_{i}\right) \int_{\Omega}\left|2  \widetilde{\nabla} x_{p} \widetilde{\nabla} u_{i}+\widetilde{\Delta} x_{p} u_{i}\right|^{2} -\left(\lambda_{k+1}-\lambda_{i}\right) \sum_{j=1}^{k} r_{i j}^{2}. \nonumber\\
  \end{align}
  Summing over $i$ from $1$ to $k$,  we have
  \begin{align}
  	&\quad\sum_{i=1}^{k}\left(\lambda_{k+1}-\lambda_{i}\right)^{2} w_{i}+2\sum_{i,j=1}^{k}\left(\lambda_{k+1}-\lambda_{i}\right)^{2} a_{i j} r_{i j} \nonumber\\
  	&\leqslant \sum_{i=1}^{k}\left(\lambda_{k+1}-\lambda_{i}\right)^{2} \int_{\Omega} x_{p} u_{i}  q_{i}+\sum_{i,j=1}^{k}\left(\lambda_{k+1}-\lambda_{i}\right)^{2} \left(\lambda_{i}-\lambda_{j}\right) a_{i j}^{2} \nonumber\\
  	&\quad+\sum_{i=1}^{k}\left(\lambda_{k+1}-\lambda_{i}\right) \int_{\Omega}\left|2 \widetilde{\nabla} x_{p} \widetilde{\nabla}u_{i}+\widetilde{\Delta} x_{p} u_{i}\right|^{2}-\sum_{i,j=1}^{k}\left(\lambda_{k+1}-\lambda_{i}\right)  r_{i j}^{2}\nonumber \\
  \end{align}
  Noting that $$ a_{i j}=a_{j i}, \ r_{i j}=-r_{j i}, $$ we infer
  \begin{align*}
  &\sum_{i,j=1}^{k}\left(\lambda_{k+1}-\lambda_{i}\right)^{2} a_{i j} r_{i j}
 \\& =\sum_{i,j=1}^k\left(\lambda_{k+1}-\lambda_{i}\right)\left(\lambda_{k+1}+\lambda_j-\lambda_j-\lambda_i\right) a_{i j} r_{i j}
  \\&=\sum_{i,j=1}^k\left(\lambda_{k+1}-\lambda_{i}\right)\left(\lambda_{k+1}-\lambda_{j}\right)a_{i j} r_{i j} +\sum_{i,j=1}^k\left(\lambda_{k+1}-\lambda_{i}\right)\left(\lambda_{j}-\lambda_{i}\right)a_{i j} r_{i j}
  \\&=\sum_{i,j=1}^k\left(\lambda_{k+1}-\lambda_{i}\right)\left(\lambda_{j}-\lambda_{i}\right)a_{i j} r_{i j},
  \end{align*}
  where we used
  $$\sum_{i,j=1}^k\left(\lambda_{k+1}-\lambda_{i}\right)\left(\lambda_{k+1}-\lambda_{j}\right)a_{i j} r_{i j}=0.$$
  Similarly,
   \begin{align*}
 &\sum_{i,j=1}^{k}\left(\lambda_{k+1}-\lambda_{i}\right)^{2} \left(\lambda_{i}-\lambda_{j}\right) a_{i j}^{2}
 \\&=\sum_{i,j=1}^{k}\left(\lambda_{k+1}-\lambda_{i}\right)\left(\lambda_{k+1}+\lambda_j-\lambda_j-\lambda_i\right)\left(\lambda_{i}-\lambda_{j}\right) a_{i j}^{2}
 \\&=\sum_{i,j=1}^k\left(\lambda_{k+1}-\lambda_{i}\right)\left(\lambda_{k+1}-\lambda_{j}\right)\left(\lambda_{i}-\lambda_{j}\right)a_{i j}^2-\sum_{i,j=1}^{k}\left(\lambda_{j}-\lambda_{i}\right)^2 a_{i j}^{2}
 \\&=-\sum_{i,j=1}^{k}\left(\lambda_{j}-\lambda_{i}\right)^2 a_{i j}^{2}.
  \end{align*}
  Therefore
  \begin{align}
  	&\quad\sum_{i=1}^{k}\left(\lambda_{k+1}-\lambda_{i}\right)^{2} w_{i} \nonumber\\
  	&\leqslant \sum_{i=1}^{k}\left(\lambda_{k+1}-\lambda_{i}\right)^{2} \int_{\Omega} x_{p} u_{i}  q_{i} +\sum_{i=1}^{k}\left(\lambda_{k+1}-\lambda_{i}\right) \int_{\Omega}\left|2  \widetilde{\nabla} x_{p} \widetilde{\nabla}u_{i}+\widetilde{\Delta} x_{p} u_{i}\right|^{2}\nonumber \\
  &-2\sum_{i,j=1}^k\left(\lambda_{k+1}-\lambda_{i}\right)\left(\lambda_{j}-\lambda_{i}\right)a_{i j} r_{i j}-\sum_{i,j=1}^{k}\left(\lambda_{j}-\lambda_{i}\right)^2 a_{i j}^{2}-\sum_{i,j=1}^{k}\left(\lambda_{k+1}-\lambda_{i}\right)  r_{i j}^{2}\nonumber
  \\&=\sum_{i=1}^{k}\left(\lambda_{k+1}-\lambda_{i}\right)^{2} \int_{\Omega} x_{p} u_{i}  q_{i} +\sum_{i=1}^{k}\left(\lambda_{k+1}-\lambda_{i}\right) \int_{\Omega}\left|2  \widetilde{\nabla} x_{p} \widetilde{\nabla}u_{i}+\widetilde{\Delta} x_{p} u_{i}\right|^{2}\nonumber \\
  &-\sum_{i=1}^{k}\left(\lambda_{k+1}-\lambda_{i}\right)[\left(\lambda_{j}-\lambda_{i}\right)a_{i j}+r_{i j}]^2\nonumber\\
  &\leqslant \sum_{i=1}^{k}\left(\lambda_{k+1}-\lambda_{i}\right)^{2} \int_{\Omega} x_{p} u_{i}  q_{i} +\sum_{i=1}^{k}\left(\lambda_{k+1}-\lambda_{i}\right) \int_{\Omega}\left|2  \widetilde{\nabla} x_{p} \widetilde{\nabla}u_{i}+\widetilde{\Delta} x_{p} u_{i}\right|^{2}.\nonumber\\
  \end{align}
  By definition of $w_{i}$, $q_{i}$ and $$\widetilde{\Delta} x_{p}=(n-2)e^{-2 f} \nabla f \nabla x_{p},$$ we have
  	\begin{align}
  		&\quad\sum_{i=1}^{k}\left(\lambda_{k+1}-\lambda_{i}\right)^{2} \int_{\Omega} x_{p} u_{i}\left(-1\right)\left(2\widetilde{\nabla}x_{p} \widetilde{\nabla}u_{i}+(n-2) e^{-2 f} \nabla f \nabla x_{p} u_{i}\right) \nonumber\\
  		&\leqslant \sum_{i=1}^{k}\left(\lambda_{k+1}-\lambda_{i}\right) \int_{\Omega}\left|2\widetilde{\nabla}x_{p} \widetilde{\nabla}u_{i}+(n-2) e^{-2 f} \nabla f \nabla x_{p} u_{i}\right|^{2},
  	\end{align}
  	which is (\ref{the first result}).
\end{proof}

\section{Proof of Proposition \ref{Yang inequality in hyperbolic space}}
\begin{proof}
In this section we consider the eigenvalues of the Dirichlet Laplacian on the conformally flat Riemannian manifold $(\mathbb{R}^n_{+}, \tilde{g})$, where
\begin{eqnarray*}
	\mathbb{R}^n_{+}:=\left\{\vec{x}=\left(x_{1}, x_{2}, \cdots, x_{n}\right) \in \mathbb{R}^{n} | x_{n}>0\right\}
\end{eqnarray*}
with the metric
\begin{eqnarray*}
\widetilde{g}=\frac{d x_{1}^{2}+\cdots+d x_{n}^{2}}{x_{n}^{t}}=e^{2 f}g,
\end{eqnarray*}
and $t$ is a real constant. In this case, the conformal factor is $$f=-\frac{t}{2}\ln{x_{n}},$$ and
\begin{align*}
\nabla f&=\left(0, \cdots, -\frac{t}{2 x_n}\right),
\\ \Delta f &=\frac{t}{2 x_{n}^{2}},
\\ e^{-2 f}&=x_{n}^{t}.
\end{align*}
First,
	\begin{align*}
	&\quad-2 \int_{\Omega} x_{p} u_{i}\widetilde{\nabla}x_{p} \widetilde{\nabla}u_{i}\\
	& = -\frac{1}{2}\int_{\Omega}\widetilde{\nabla} x_{p}^{2} \widetilde{\nabla}u_{i}^{2}\\
	& = \frac{1}{2}\int_{\Omega}\widetilde{\Delta} x_{p}^{2} u_{i}^{2}. \\
\end{align*}
Since
\begin{align}
	\widetilde{\Delta}\left(x_{p}^{2}\right) & =e^{-2 f}\left(\Delta\left(x_{p}^{2}\right)+(n-2) \nabla f \nabla\left(x_{p}^{2}\right)\right) \nonumber\\
	& =e^{-2 f}\left(2\left|\nabla x_{p}\right|^{2}+2(n-2) \nabla f \nabla x_{p} x_{p}\right) \nonumber\\
	& =x_{n}^{t}\left(2-t(n-2) \delta_{p n}\right),\nonumber
	\end{align}
we obtain
	\begin{align}\label{equ 1}
	&-2\sum_{p=1}^{n} \int_{\Omega} x_{p} u_{i}\widetilde{\nabla}x_{p} \widetilde{\nabla}u_{i}\nonumber\\
	& = \frac{1}{2}\sum_{p=1}^{n}\int_{\Omega}\widetilde{\Delta} x_{p}^{2} u_{i}^{2} \nonumber\\
	& =\left(n-\frac{\left(n-2\right)t}{2} \right)\int_{\Omega} x_{n}^{t}u_{i}^{2}.
\end{align}
Direct computations show that
\begin{align}\label{equ 2}
		&-(n-2) \sum_{p=1}^{n}\int_{\Omega} x_{p} u_{i} e^{-2 f} \nabla f \nabla x_{p} u_{i} \nonumber\\
		& = \frac{(n-2)t}{2} \sum_{p=1}^{n}\int_{\Omega} \delta_{p n} x_{n}^{t} u_{i}^{2} \nonumber\\
		& = \frac{(n-2)t}{2}\int_{\Omega} x_{n}^{t} u_{i}^{2}
\end{align}
and
  \begin{align}\label{equ 3}
  	&\sum_{p= 1}^{n} \int_{\Omega}\left|(n-2) e^{-2 f} \nabla f \nabla x_{p} u_{i} \right|^{2} \nonumber\\
  	 & = \sum_{p = 1}^{n} \int_{\Omega}\left|(n-2) x_{n}^{t}\delta_{p n}\left(-\frac{t}{2x_{p}}\right) u_{i}\right|^{2} \nonumber\\
  	 & = (n-2)^{2}\int_{\Omega}\left|x_{n}^{t} \left(-\frac{t}{2x_{n}}\right) u_{i}\right|^{2}\nonumber \\
  	 &= \frac{(n-2)^{2} t^{2}}{4}\int_{\Omega} x_{n}^{2 t-2} u_{i}^{2}.
  \end{align}

  Moreover,
	\begin{align}
		& \quad\sum_{p=1}^{n} 4(n-2)\int_{\Omega}\widetilde{\nabla} x_{p} \widetilde{\nabla} u_{i}e^{-2 f} \nabla f \nabla x_{p} u_{i}  \nonumber\\
	&=\sum_{p=1}^{n} 4(n-2) \int_{\Omega} e^{-2 f} \nabla x_{p} \nabla u_{i} e^{-2 f} \nabla f \nabla x_{p} u_{i} \nonumber\\
	&=4(n-2) \int_{\Omega} e^{-4 f} \nabla u_{i} \nabla f u_{i} \nonumber\\
	&=4(n-2) \int_{\Omega} e^{-2 f} \widetilde{\nabla} u_{i} \widetilde{\nabla} f u_{i} \nonumber\\
	&=2(n-2) \int_{\Omega} e^{-2 f} \widetilde{\nabla} f \widetilde{\nabla} u_{i}^{2} \nonumber\\
	&=-2(n-2) \int_{\Omega} \widetilde{\nabla}\cdot\left(e^{-2 f} \widetilde{\nabla} f\right) u_{i}^{2},
	\end{align}
and
\begin{align}\label{compute}
	& \widetilde{\nabla}\cdot\left(e^{-2 f} \widetilde{\nabla} f\right) \nonumber\\
	&=  -2 e^{-2 f} \widetilde{\nabla} f  \widetilde{\nabla} f+e^{-2 f} \widetilde{\Delta f} \nonumber\\
	 &= -2 e^{-4 f} \nabla f \nabla f+e^{-2 f} \left(e^{-2 f}(\Delta f+(n-2) \nabla f \nabla f)\right) \nonumber\\
	& = -2  x_{n}^{2 t} \frac{t^{2}}{4 x_{n}^{2}}+x_{n}^{2 t}\left(\frac{t}{2 x_{n}^{2}}+(n-2) \frac{t^{2}}{4 x_{n}^{2}}\right) \nonumber\\
	&= \frac{(n-4) t^{2}+2 t}{4} x_{n}^{2 t-2}.
\end{align}
Therefore we have
\begin{align}\label{equ 4}
		& \quad\sum_{p=1}^{n} 4(n-2)\int_{\Omega}\widetilde{\nabla} x_{p} \widetilde{\nabla} u_{i}e^{-2 f} \nabla f \nabla x_{p} u_{i}  \nonumber\\
		&=-\frac{(n-2)(n-4)t^2+2(n-2)t}{2}\int_{\Omega}x_{n}^{2t-2}u_{i}^{2}.
\end{align}
In the end,
		\begin{align*}
			& \sum_{p=1}^{n} \int_{\Omega}\left|2 \widetilde{\nabla} x_{p} \widetilde{\nabla} u_{i}\right|^{2} \\
			&=  4 \sum_{p=1}^{n} \int_{\Omega}\left|e^{-2 f} \nabla x_{p} \nabla u_{i}\right|^{2} \\
			&=  4 \int_{\Omega} e^{-4 f}\left|\nabla u_{i}\right|^{2} \\
			&= 4 \int_{\Omega} e^{-4 f} \nabla u_{i} \nabla u_{i} \\
			&=  4 \int_{\Omega} e^{-2 f} \widetilde{\nabla} u_{i} \widetilde{\nabla} u_{i}\\
			&=-4\int_{\Omega} \widetilde{\nabla}\cdot\left(e^{-2 f} \widetilde{\nabla} u_{i}\right) u_{i} \\
			&=-4\int_{\Omega}-2 e^{-2 f} \widetilde{\nabla} f \widetilde{\nabla} u_{i}  u_{i}-4\int_{\Omega} e^{-2 f} \widetilde{\Delta} u_{i}  u_{i} \\
			&= 8 \int_{\Omega} e^{-2 f} \widetilde{\nabla} f \widetilde{\nabla} u_{i} u_{i}+4\lambda_{i} \int_{\Omega} x_{n}^{t}  u_{i}^{2}.\\
		\end{align*}
Since
\begin{align*}
	& \int_{\Omega} e^{-2 f} \widetilde{\nabla} f \widetilde{\nabla} u_{i} u_{i} \\
	= & \frac{1}{2} \int_{\Omega} e^{-2 f} \widetilde{\nabla} f \widetilde{\nabla} u_{i}^{2} \\
	= & -\frac{1}{2} \int_{\Omega} \widetilde{\nabla}\cdot\left(e^{-2 f} \widetilde{\nabla} f\right) u_{i}^{2},
\end{align*}
	 from (\ref{compute}) we have
	 \begin{align}\label{equ 5}
	 		&\quad \sum_{p=1}^{n} \int_{\Omega}\left|2 \widetilde{\nabla} x_{p} \widetilde{\nabla} u_{i}\right|^{2} \nonumber\\
	 		&=\left((4-n) t^{2}-2 t\right)\int_{\Omega} x_{n}^{2 t-2} u_{i}^{2}+4\lambda_{i} \int_{\Omega} x_{n}^{t}  u_{i}^{2}.
	 \end{align}
From (\ref{equ 1}), (\ref{equ 2}), (\ref{equ 3}), (\ref{equ 4}) and (\ref{equ 5}) we obtain
	 \begin{align}
		& \sum_{i=1}^{k}\left(\lambda_{k+1}-\lambda_{i}\right)^{2}\left(n \int_{\Omega} x_{n}^{t} u_{i}^{2}\right) \nonumber\\
		\leqslant & \sum_{i=1}^{k}\left(\lambda_{k+1}-\lambda_{i}\right)\left(4 \lambda_{i} \int_{\Omega} x_{n}^{t} u_{i}^{2}+\left(\left(-\frac{n^2}{4}+n+1\right)t^{2}-nt\right)\int_{\Omega} x_{n}^{2 t-2} u_{i}^{2}\right),
	\end{align}
	which proves (\ref{the second result}).
\end{proof}

\section{Proof of Proposition \ref{Yang type inequality for radial conformal factor} and Corollary \ref{the second proof}}
Poof of Proposition \ref{Yang type inequality for radial conformal factor}.

\begin{proof} In this section, we consider  an $n$-dimensional conformally flat manifold $M$ with metric $\widetilde{g}=e^{2f}g$, where $f$ is a radial symmetric function. In this case, we have
\begin{eqnarray*}
	\nabla f=\left(f^{\prime}(r) \frac{x_{1}}{|x|}, \cdots, f^{\prime}(r) \frac{x_{n}}{|x|}\right), \ \Delta f=f^{\prime \prime}(r)+\frac{f^{\prime}(r)(n-1)}{r},
\end{eqnarray*}
where $r:=|x|.$

As in the proof of Proposition \ref{Yang inequality in hyperbolic space}, first we have
		\begin{align*}
		&\quad-2 \int_{\Omega} x_{p} u_{i}\widetilde{\nabla}x_{p} \widetilde{\nabla}u_{i}\\
		& = \frac{1}{2}\int_{\Omega}\widetilde{\Delta} x_{p}^{2} u_{i}^{2}. \\
	\end{align*}
For each $1\leq p\leq n$, since
	\begin{align}
		\widetilde{\Delta}\left(x_{p}^{2}\right) & =e^{-2 f}\left(\Delta\left(x_{p}^{2}\right)+(n-2) \nabla f \nabla\left(x_{p}^{2}\right)\right) \nonumber\\
		& =e^{-2 f}\left(2\left|\nabla x_{p}\right|^{2}+2(n-2) \nabla f \nabla x_{p} x_{p}\right) \nonumber\\
		& =e^{-2 f}\left(2+2(n-2) f^{\prime}(r) \frac{x_{p}^{2}}{|x|}\right), \nonumber\\
	\end{align}
	we obtain
	\begin{align}\label{equ1}
		&-2\sum_{p=1}^{n} \int_{\Omega} x_{p} u_{i}\widetilde{\nabla}x_{p} \widetilde{\nabla}u_{i}\nonumber\\
		& = \frac{1}{2}\sum_{p=1}^{n}\int_{\Omega}\widetilde{\Delta} x_{p}^{2} u_{i}^{2} \nonumber\\
		& =n \int_{\Omega} e^{-2 f} u_{i}^{2}+(n-2) \int_{\Omega} e^{-2 f} f^{\prime}(r) |x|  u_{i}^{2}.
	\end{align}
Direct computations show that
\begin{align}\label{equ2}
	&-(n-2) \sum_{p=1}^{n}\int_{\Omega} x_{p} u_{i} e^{-2 f} \nabla f \nabla x_{p} u_{i} \nonumber\\
& = -(n-2) \sum_{p=1}^{n}\int_{\Omega} e^{-2 f} \nabla f \nabla x_{p} x_{p}  u_{i}^{2}  \nonumber\\
& = -(n-2) \int_{\Omega} e^{-2 f} f^{\prime}(r) |x|  u_{i}^{2}
\end{align}
and
\begin{align}\label{equ3}
&\sum_{p= 1}^{n} \int_{\Omega}\left|(n-2) e^{-2 f} \nabla f \nabla x_{p} u_{i} \right|^{2} \nonumber\\
& =(n-2)^{2} \sum_{p=1}^{n} \int_{\Omega} e^{-4 f} \left|f^{\prime}(r)\right|^{2} \frac{x_{p}^{2}}{|x|^{2}}  u_{i}^{2} \nonumber\\
& =(n-2)^{2} \int_{\Omega} e^{-4 f}\left|f^{\prime}(r)\right|^{2}  u_{i}^{2}.
\end{align}
Note that
\begin{align}\label{compute*1}
	& \quad\sum_{p=1}^{n} 4(n-2)\int_{\Omega}\widetilde{\nabla} x_{p} \widetilde{\nabla} u_{i}e^{-2 f} \nabla f \nabla x_{p} u_{i}  \nonumber\\
	&=\sum_{p=1}^{n} 4(n-2) \int_{\Omega} e^{-2 f} \nabla x_{p} \nabla u_{i} e^{-2 f} \nabla f \nabla x_{p} u_{i} \nonumber\\
	&=4(n-2) \int_{\Omega} e^{-4 f} \nabla u_{i} \nabla f u_{i} \nonumber\\
&=2(n-2)\int_{\Omega} e^{-4 f} \nabla u^2_{i} \nabla f \nonumber\\
&=2(n-2)\int_{\Omega}e^{-2 f} \widetilde{\nabla} u^2_{i}\widetilde{\nabla} f \nonumber\\
	&=-2(n-2) \int_{\Omega} \widetilde{\nabla}\cdot\left(e^{-2 f} \widetilde{\nabla} f\right) u_{i}^{2},
\end{align}
and
\begin{align}\label{compute*2}
\widetilde{\nabla}\cdot\left(e^{-2 f} \widetilde{\nabla} f\right)\nonumber &=  -2 e^{-2 f} \widetilde{\nabla} f  \widetilde{\nabla} f+e^{-2 f} \widetilde{\Delta f}\nonumber \\
&= -2 e^{-4 f} \nabla f \nabla f+e^{-2 f} \left(e^{-2 f}(\Delta f+(n-2) \nabla f \nabla f)\right) \nonumber\\
& =(n-4) e^{-4f} \left |\nabla f \right |^{2}+e^{-4 f} \Delta f\nonumber\\
&=e^{-4 f}\left((n-4) \left|f^{\prime}(r)\right|^{2}+f^{\prime \prime}(r)+\frac{f^{\prime}(r) (n-1)}{r}\right).
\end{align}
Therefore we get from (\ref{compute*1}) and (\ref{compute*2}) that

\begin{align}\label{equ4}
& \quad\sum_{p=1}^{n} 4(n-2)\int_{\Omega}\widetilde{\nabla} x_{p} \widetilde{\nabla} u_{i}e^{-2 f} \nabla f \nabla x_{p} u_{i}  \nonumber\\
&=-2(n-2)(n-4)\int_{\Omega}e^{-4 f}\left|f^{\prime}(r)\right|^{2} u_{i}^{2}
\nonumber\\&-2(n-2)\int_{\Omega}e^{-4 f}\left(f^{\prime \prime}(r)+\frac{f^{\prime}(r) (n-1)}{r}\right)u_{i}^{2}.
\end{align}
Finally,
	\begin{align*}
	& \sum_{p=1}^{n} \int_{\Omega}\left|2 \widetilde{\nabla} x_{p} \widetilde{\nabla} u_{i}\right|^{2}
\\&=  4 \sum_{p=1}^{n} \int_{\Omega}\left|e^{-2 f} \nabla x_{p} \nabla u_{i}\right|^{2} \\
	&=  4 \int_{\Omega} e^{-4 f}\left|\nabla u_{i}\right|^{2}
\\&= 4 \int_{\Omega} e^{-4 f} \nabla u_{i} \nabla u_{i} \\
	&=  4 \int_{\Omega} e^{-2 f} \widetilde{\nabla} u_{i} \widetilde{\nabla} u_{i}
\\&=-4\int_{\Omega} \widetilde{\nabla}\cdot\left(e^{-2 f} \widetilde{\nabla} u_{i}\right) u_{i} \\
	&=-4\int_{\Omega}-2 e^{-2 f} \widetilde{\nabla} f \widetilde{\nabla} u_{i}  u_{i}-4\int_{\Omega} e^{-2 f} \widetilde{\Delta} u_{i}  u_{i}
\\&= 8 \int_{\Omega} e^{-2 f} \widetilde{\nabla} f \widetilde{\nabla} u_{i} u_{i}+4\lambda_{i} \int_{\Omega}e^{-2f}  u_{i}^{2}.
\end{align*}
Since
\begin{align*}
\int_{n} e^{-2 f} \widetilde{\nabla} f \widetilde{\nabla} u_{i} u_{i}
= \frac{1}{2} \int_{n} e^{-2 f} \widetilde{\nabla} f \widetilde{\nabla} u_{i}^{2}
=  -\frac{1}{2} \int_{n} \widetilde{\nabla}\cdot\left(e^{-2 f} \widetilde{\nabla} f\right) u_{i}^{2},
\end{align*}
by (\ref{compute*2}) we obtain
\begin{align}\label{equ5}
&\sum_{p=1}^{n} \int_{\Omega}\left|2 \widetilde{\nabla} x_{p} \widetilde{\nabla} u_{i}\right|^{2} \nonumber \\
&=  -4 \int_{\Omega} \widetilde{\nabla}\cdot\left(e^{-2 f} \widetilde{\nabla} f\right) u_{i}^{2}+4 \lambda_{i} \int_{\Omega} e^{-2 f} u_{i}^{2} \nonumber\\
 &= -4(n-4) \int_{\Omega} e^{-4 f}\left|f^{\prime}(r)\right|^{2} u_{i}^{2}-4 \int_{\Omega} e^{-4 f}\left(f^{\prime \prime}(r)+\frac{f^{\prime}(r)(n-1)}{r}\right) u_{i}^{2}\nonumber \\
& \quad+4 \lambda_{i} \int_{\Omega} e^{-2 f} u_{i}^{2}.
\end{align}
From (\ref{equ1}), (\ref{equ2}), (\ref{equ3}), (\ref{equ4}), (\ref{equ5}), we obtain
\begin{align}
	& \sum_{i=1}^{k}\left(\lambda_{k+1}-\lambda_{i}\right)^{2}\left(n \int_{\Omega} e^{-2 f} u_{i}^{2}\right) \nonumber\\
	\leqslant & \sum_{i=1}^{k}\left(\lambda_{k+1}-\lambda_{i}\right)\left(4 \lambda_{i} \int_{\Omega} e^{-2f} u_{i}^{2}-\left(n^{2}-4 n-4\right) \int_{\Omega} e^{-4 f}\left|f^{\prime}(r)\right|^{2} u_{i}^{2}\right. \nonumber\\
	& \left.-2 n \int_{\Omega} e^{-4 f}\left(f^{\prime\prime}(r)+\frac{f^{\prime}(r)(n-1)}{r}\right) u_{i}^{2}\right),
\end{align}
which proves (\ref{the radial conformal factor case}).
\end{proof}

Proof of Corollary \ref{the second proof}.

\begin{proof}
Now let $(M^{n}, \tilde{g})$ be the hyperbolic space with the Poincar\'e unit disk  model, and $\Omega$ a bounded smooth domain in $\mathbb{H}^n(-1)$. In this case the conformal factor is the radial symmetric function $$f(r)=\ln\frac{2}{1-r^2}.$$ Then
$$f'(r)=\frac{2r}{1-r^2}, \ f''(r)=\frac{2r^2+2}{(1-r^2)^2}.$$
Insert these formulas into (\ref{the radial conformal factor case}), we get
\begin{align*}
& \sum_{i=1}^{k}\left(\lambda_{k+1}-\lambda_{i}\right)^{2}\left(n \int_{\Omega}\left(1-|x|^{2}\right)^{2} u_{i}^{2}\right) \nonumber\\
		 	\leqslant & \sum_{i=1}^{k}\left(\lambda_{k+1}-\lambda_{i}\right)\big[4\lambda_{i} \int_{\Omega}\left(1-|x|^{2}\right)^{2}u_{i}^{2}-{n}^{2}\int_{\Omega}\left(1-|x|^{2}\right)^{2}u_{i}^{2} \nonumber\\
&+\left(2n+4\right) \int_{\Omega}\left(1-|x|^{2}\right)^{2}|x|^{2} u_{i}^{2}\big]
\\\leqslant & \sum_{i=1}^{k}\left(\lambda_{k+1}-\lambda_{i}\right)\big[4\lambda_{i} \int_{\Omega}\left(1-|x|^{2}\right)^{2}u_{i}^{2}-{n}^{2}\int_{\Omega}\left(1-|x|^{2}\right)^{2}u_{i}^{2} \nonumber\\
&+\left(2n+4\right) \int_{\Omega}\left(1-|x|^{2}\right)^{2}u_{i}^{2}\big]
\\& =\sum_{i=1}^{k}\left(\lambda_{k+1}-\lambda_{i}\right)\big[4\lambda_{i} \int_{\Omega}\left(1-|x|^{2}\right)^{2}u_{i}^{2}-\left({n}^{2}-2n-4\right)\int_{\Omega}\left(1-|x|^{2}\right)^{2}u_{i}^{2}\big].
\end{align*}
where in the last inequality we used $|x|<1$. 

Finally let $k=1$  and divided by the term $\int_{\Omega}\left(1-|x|^{2}\right)^{2}u_{i}^{2}$ on both sides of the above inequality we get (\ref{the third result}).

\end{proof}

\section{More results}

Universal inequalities for the eigenvalues of the Dirichlet Laplacian has attracted much attention in the past decades.  Other than those mentioned in the introduction part, various universal inequalities has also been obtained by mathematicians for the eigenvalues of the Dirichlet Laplacian on a compact homogeneous manifold, or a compact minimal submanifold in a sphere, or a submanifold in the Euclidean space (for example \cite{Chen,CY1,CY2,Ha, HM, HS,Le,Li,LZ,PY} and so on). In this section we will see that universal inequalities for more special conformally flat manifolds could be obtained as consequences of Propositions \ref{theorem 1} and \ref{Yang type inequality for radial conformal factor}. First we have
\begin{pro}\label{more results 1}
For a conformally flat manifold $(\mathbb{R}_{+}^{n},\widetilde{g})$ with conformal metric $$\widetilde{g}=\frac{d x_{1}^{2}+d x_{2}^{2}+\cdots+d x_{n}^{2}}{x_{n}^{t}},$$ assume that 
\\
\\ \rm (i) $n\geq5$ and $t\geq 0$, or 
\\
\\ \rm (ii) $n=2,3,4, \ and \ 0\leq t\leq\frac{n}{-\frac{n^2}{4}+n+1},$ or 
\\
\\ \rm (iii) $n\geq 5 \ and \ t\leq\frac{n}{-\frac{n^2}{4}+n+1}.$ 
\\
\\Then eigenvalues $\lambda_{i}$'s of the eigenvalue problem (\ref{Dirichlet problem}) satisfy
     \begin{align}\label{more 1}
     	\lambda_2-\lambda_1
     \leqslant \frac{4}{n}\lambda_{1}.
     \end{align}
     Here $\mathbb{R}_{+}^n:=\{(x_1,\cdots,x_n)\in \mathbb{R}^n | x_n>0\}$.
\end{pro}
\begin{proof}
If $t\geq0$, direct computations show that $-\frac{n^2}{4}+n+1\leq0$  when $n\geq5$, and thus $(-\frac{n^2}{4}+n+1)t^2-nt\leq0$; when $2\leq n\leq 4$,  we have also $(-\frac{n^2}{4}+n+1)t^2-nt\leq0$ if $0\leq t\leq\frac{n}{-\frac{n^2}{4}+n+1}$. If $t<0$, then  when $n\geq5$ and $t\leq\frac{n}{-\frac{n^2}{4}+n+1}$ we still have $(-\frac{n^2}{4}+n+1)t^2-nt\leq0$. Therefore under assumptions of the above proposition, from (\ref{the second result}) we have
\begin{align}\label{the second result'}
     	& \sum_{i=1}^{k}\left(\lambda_{k+1}-\lambda_{i}\right)^{2}\left(n \int_{\Omega} x_{n}^{t} u_{i}^{2}\right) \nonumber\\
     \leqslant & \sum_{i=1}^{k}\left(\lambda_{k+1}-\lambda_{i}\right)4 \lambda_{i} \int_{\Omega} x_{n}^{t} u_{i}^{2},
     \end{align}
which implies (\ref{more 1}) by letting $k=1$.
\end{proof}

Next we consider conformally flat manifolds $(\mathbb{B}^n, \frac{4}{(1-|x|^2)^t}g)$, where $t$ is a real constant. In this case $e^{2f}=\frac{4}{(1-|x|^2)^t}$, while $f(r)=\ln\frac{2}{(1-r^2)^\frac{t}{2}}$. Then
\begin{align*}
f'(r)=\frac{tr}{1-r^2}, \ f^{''}(r)=\frac{tr^2+t}{(1-r^2)^2}.
\end{align*}
Then from (\ref{the radial conformal factor case}) we obtain
\begin{align}\label{radial case'}
		 	& \sum_{i=1}^{k}\left(\lambda_{k+1}-\lambda_{i}\right)^{2}\left(n \int_{\Omega}\left(1-|x|^{2}\right)^{t} u_{i}^{2}\right) \nonumber\\
		 	\leqslant & \sum_{i=1}^{k}\left(\lambda_{k+1}-\lambda_{i}\right)\big[4\lambda_{i} \int_{\Omega}\left(1-|x|^{2}\right)^{t}u_{i}^{2}-\frac{t{n}^{2}}{2}\int_{\Omega}\left(1-|x|^{2}\right)^{t}u_{i}^{2} \nonumber\\
&+[-\frac{n^2-4n-4}{4}t^2+\frac{t(n-2)n}{2}] \int_{\Omega}\left(1-|x|^{2}\right)^{2}|x|^{2} u_{i}^{2}\big].
		 \end{align}
Then similarly with the proof of Proposition \ref{more results 1} we get that if $0\leq t\leq\frac{2(n-2)}{n^2-4n-4} (n\geq 5)$ or $t\leq  \frac{2(n-2)}{n^2-4n-4}(2\leq n\leq 4)$, we have
$$\left[-\frac{n^2-4n-4}{4}t^2+\frac{t(n-2)n}{2}\right]\geq0.$$ Then from (\ref{radial case'}) we obtain that
\begin{align*}
		 	& \sum_{i=1}^{k}\left(\lambda_{k+1}-\lambda_{i}\right)^{2}\left(n \int_{\Omega}\left(1-|x|^{2}\right)^{t} u_{i}^{2}\right) \nonumber\\
		 	\leqslant & \sum_{i=1}^{k}\left(\lambda_{k+1}-\lambda_{i}\right)\big[4\lambda_{i} \int_{\Omega}\left(1-|x|^{2}\right)^{t}u_{i}^{2}
+\left(-\frac{n^2-4n-4}{4}t^2-nt\right) \int_{\Omega}\left(1-|x|^{2}\right)^{t} u_{i}^{2}\big],
 \end{align*}
 which implies
\begin{align*}
		 	 \lambda_2-\lambda_1\leqslant & \frac{4}{n}\big[\lambda_1+\left(-\frac{n^2-4n-4}{16}t^2-\frac{nt}{4}\right)\big].
 \end{align*}

 To summarize, we obtain
 \begin{pro}
For a comformally flat Riemannian manifold $(\mathbb{B}^n, \tilde{g})$ with
$$\tilde{g}=\frac{4}{(1-r^2)^t}g,$$
assume that 
\\
\\ \rm (i) $0\leq t\leq\frac{2(n-2)}{n^2-4n-4} (n\geq 5),$ or 
\\
\\ \rm (ii) $t\leq \frac{2(n-2)}{n^2-4n-4}(2\leq n\leq 4).$
\\
 \\ Then eigenvalues $\lambda_{i}$'s of the eigenvalue problem (\ref{Dirichlet problem}) satisfy
 \begin{align*}
		 	 \lambda_2-\lambda_1\leqslant & \frac{4}{n}\big[\lambda_1+\left(-\frac{n^2-4n-4}{16}t^2-\frac{nt}{4}\right)\big].
 \end{align*}
 \end{pro}

\vspace{1 cm}
\textbf{Declarations.} 
\\
\\ \textbf{Conflict of interest.} The authors hereby state that there are no conflicts of
interest regarding the presented results.
\\
\\ \textbf{Data availability statements.} No data sets were generated or analysed during
the current study.

\vspace{1 cm}
\textbf{Acknowledgement.} This paper is supported by the Natural Science Foundation of China (Grant no. 12271069). The first author would like to thank Professor Qingming Cheng for his interests and comments on this paper.
	{}
\vspace{1cm}\sc
	
Yong Luo

Mathematical Science Research Center of Mathematics,

Chongqing University of Technology,

Chongqing, 400054, China

{\tt yongluo-math@cqut.edu.cn}

\vspace{1cm}\sc
Xianjing Zheng

Mathematical Science Research Center of Mathematics,

Chongqing University of Technology,

Chongqing, 400054, China

{\tt 1311169535@qq.com}
\end{document}